\documentclass{llncs}
\usepackage[latin1]{inputenc}
\usepackage{enumerate,enumitem}
\usepackage{graphicx}
\usepackage{color}

\usepackage{amsmath,amssymb}
\let\doendproof\endproof
\renewcommand\endproof{~\hfill\qed\doendproof}

\usepackage[unicode=true]{hyperref}

\DeclareGraphicsExtensions{.pdf,.png,.eps}
\graphicspath{{pics/}}

\newcommand{\qn}{\ensuremath{\operatorname{qn}}}
\let\leq\leqslant
\let\geq\geqslant

\let\epsilon\varepsilon

\let\setminus\smallsetminus

\DeclareMathOperator{\width}{width}

\title{The Queue-Number of Posets of Bounded Width or Height}
\author{Kolja Knauer\inst{1}\thanks{Kolja Knauer was supported by ANR projects GATO: ANR-16-CE40-0009-01 and DISTANCIA: ANR-17-CE40-0015}, Piotr Micek\inst{2}
\thanks{Piotr Micek was partially supported by the National Science Center of Poland under grant no.\ 2015/18/E/ST6/00299.}
and Torsten Ueckerdt\inst{3}}

\institute{Aix Marseille Univ, Universit\'e de Toulon, CNRS, LIS, Marseille, France \email{kolja.knauer@lis-lab.fr} \and  Jagiellonian University,
Faculty of Mathematics and Computer Science, Theoretical Computer Science Department, Poland
\email{piotr.micek@tcs.uj.edu.pl} 
\and Karlsruhe Institute of Technology (KIT), 
Institute of Theoretical Informatics, Germany
\email{torsten.ueckerdt@kit.edu}}

\begin{document}

\maketitle

\begin{abstract}
 Heath and Pemmaraju~\cite{HP-97} conjectured that the queue-number of a poset is bounded by its width and if the poset is planar then also by its height. We show that there are planar posets whose queue-number is larger than their height, refuting the second conjecture. On the other hand, we show that any poset of width $2$ has queue-number at most $2$, thus confirming the first conjecture in the first non-trivial case. Moreover, we improve the previously best known bounds and show that planar posets of width $w$ have queue-number at most $3w-2$ while any planar poset with $0$ and $1$ has queue-number at most its width.
\end{abstract}

\section{Introduction}
\label{sec:introduction}

A \emph{queue layout} of a graph consists of a total ordering on its vertices
and an assignment of its edges to \emph{queues}, such that no two edges in a single queue are nested.
The minimum number of queues needed in a queue layout of a graph $G$ is called its \emph{queue-number} and denoted by $\qn(G)$.

To be more precise, let $G$ be a graph and let $L$ be a linear order on the vertices of $G$.
We say that the edges $uv,u'v'\in E(G)$ are \emph{nested} with respect to $L$ if
$u< u' < v'<v$ or $u'<u<v<v'$ in $L$.
Given a linear order $L$ of the vertices of $G$, the edges $u_1v_1,\ldots, u_kv_k$ of $G$ form a \emph{rainbow} of size $k$ if $u_1 <\cdots<u_k<v_k<\cdots<v_1$ in $L$.
Given $G$ and $L$, the edges of $G$ can be partitioned into $k$ queues if and only if there is no rainbow of size $k+1$ in $L$, see~\cite{HR-92}.

The queue-number was introduced by Heath and Rosenberg in 1992~\cite{HR-92} as an analogy to book embeddings.
Queue layouts were implicitly used before and have applications in fault-tolerant processing, sorting with parallel queues, matrix computations, scheduling parallel processes, and communication management in distributed algorithm (see~\cite{HLR-92,HR-92,NOdMW12}).

Perhaps the most intriguing question concerning queue-numbers is whether planar graphs have bounded queue-number.
\begin{conjecture}[Heath and Rosenberg~\cite{HR-92}]\label{conj:queue-planar-graphs}{\ \\}
 The queue-number of planar graphs is bounded by a constant.
\end{conjecture}

In this paper we study queue-numbers of posets.
The parameter was introduced in 1997 by Heath and Pemmaraju~\cite{HP-97} and the main idea is that given a poset one should lay it out respecting its relation.
Two elements $a,b$ of a poset are called \emph{comparable} if $a < b$ or $b < a$, and \emph{incomparable}, denoted by $a \parallel b$, otherwise.
Posets are visualized by their \emph{diagrams}:
Elements are placed as points in the plane and whenever $a<b$ in the poset, and there is no element $c$ with $a<c<b$,
there is a curve from $a$ to $b$ going upwards (that is $y$-monotone).
We denote this case as $a\prec b$.
The diagram represents those relations which are essential in the sense that they are not implied
by transitivity, also known as \emph{cover relations}.
The undirected graph implicitly defined by such a diagram is the \emph{cover graph} of the poset.
Given a poset $P$, a \emph{linear extension} $L$ of $P$ is a linear order on the elements of $P$ such that $x <_L y$, whenever $x <_P y$. (Throughout the paper we use a subscript on the symbol $<$, if we want to emphasize which order it represents.)
Finally, the \emph{queue-number of a poset} $P$, denoted by $\qn(P)$, is the smallest $k$ such that there is a linear extension $L$ of $P$ for which the resulting linear layout of $G_P$ contains no $(k+1)$-rainbow.
Clearly we have $\qn(G_P)\leq \qn(P)$, i.e., the queue-number of a poset is at least the queue-number of its cover graph. 
It is shown in~\cite{HP-97} that even for \emph{planar posets}, that is posets admitting crossing-free diagrams, there is no function $f$ such that $\qn(P)\leq f(\qn(G_P))$.
\begin{figure}
  \centering
  \includegraphics{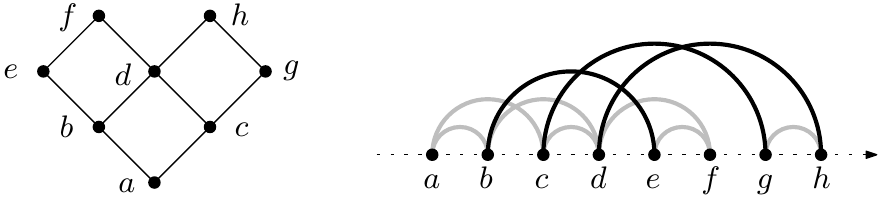}
  \caption{A poset and a layout with two queues (gray and black). Note that the order of the elements on the spine is a linear extension of the poset.}
 \end{figure}

Heath and Pemmaraju~\cite{HP-97} investigated the maximum queue-number of several classes of posets, in particular with respect to bounded width (the maximum number of pairwise incomparable elements) and height (the maximum number of pairwise comparable elements). 
A set with every two elements being comparable is a \emph{chain}.
A set with every two distinct elements being incomparable is an \emph{antichain}.
They proved that if $\width(P)\leq w$, then $\qn(P)\leq w^2$.
The lower bound is attained by \emph{weak orders}, i.e., chains of antichains and is conjectured to be the upper bound as well:
\begin{conjecture}[Heath and Pemmaraju~\cite{HP-97}]\label{conj:general-width}{\ \\}
 Every poset of width $w$ has queue-number at most $w$.
\end{conjecture}
Furthermore, they made a step towards this conjecture for planar posets:
if a planar poset $P$ has $\width(P)\leq w$, then $\qn(P)\leq 4w-1$.
For the lower bound side they provided planar posets of width $w$ and queue-number $\lceil\sqrt{w}\rceil$.

We improve the bounds for planar posets and get the following:
\begin{theorem}\label{thm:planar-width}
 Every planar poset of width $w$ has queue-number at most $3w-2$.
 Moreover, there are planar posets of width $w$ and queue-number $w$.
\end{theorem}

As an ingredient of the proof we show that posets without certain subdivided crowns satisfy Conjecture~\ref{conj:general-width} (c.f.\ Theorem~\ref{thm:subdivided-crowns}).
This implies the conjecture for interval orders and planar posets with (unique minimum) 0 and (unique maximum) 1 (c.f.\ Corollary~\ref{cor:interval}). 
Moreover, we confirm Conjecture~\ref{conj:general-width} for the first non-trivial case $w=2$:
\begin{theorem}\label{thm:general-width-2}
 Every poset of width $2$ has queue-number at most $2$.
\end{theorem}

An easy corollary of this is that all posets of width $w$ have queue-number at most $w^2-w+1$ (c.f.\ Corollary~\ref{cor:general-width}).

Another conjecture of Heath and Pemmaraju concerns planar posets of bounded height:
\begin{conjecture}[Heath and Pemmaraju~\cite{HP-97}]\label{conj:planar-height}{\ \\}
 Every planar poset of height~$h$ has queue-number at most $h$.
\end{conjecture}

We show that Conjecture~\ref{conj:planar-height} is false for the first non-trivial case $h=2$:
\begin{theorem}\label{thm:planar-height-2}
 There is a planar poset of height $2$ with queue-number at least $4$.
\end{theorem}

Furthermore, we establish a link between a relaxed version of Conjecture~\ref{conj:planar-height} and Conjecture~\ref{conj:queue-planar-graphs}, namely we show that the latter is equivalent to planar posets of height $2$ having bounded queue-number (c.f.\ Theorem~\ref{thm:equivalent-conjectures}).
On the other hand, we show that Conjecture~\ref{conj:planar-height} holds for planar posets with 0 and 1:
\begin{theorem}\label{thm:planar-height-0-and-1}
 Every planar poset of height $h$ with $0$ and $1$ has queue-number at most $h-1$.
\end{theorem}

\paragraph{Organization of the paper.}
In Section~\ref{sec:general-posets} we consider general (not necessarily planar) posets and give upper bounds on their queue-number in terms of their width, such as Theorem~\ref{thm:general-width-2}.
In Section~\ref{sec:planar-posets-width} we consider planar posets and bound the queue-number in terms of the width, both from above and below, i.e., we prove Theorem~\ref{thm:planar-width}.
In Section~\ref{sec:planar-posets-height} we give a counterexample to Conjecture~\ref{conj:planar-height} by constructing a planar poset with height $2$ and queue-number at least $4$.
Here we also argue that proving \emph{any} upper bound on the queue-number of such posets is equivalent to proving Conjecture~\ref{conj:queue-planar-graphs}. Finally, we show that Conjecture~\ref{conj:planar-height} holds for planar posets with 0 and 1 and that for every $h$ there is a planar poset of height $h$ and queue-number $h-1$ (c.f.\ Proposition~\ref{prop:planar-height-LB}).

\section{General Posets of Bounded Width}
\label{sec:general-posets}

By Dilworth's Theorem~\cite{D-50}, the width of a poset $P$ coincides with the smallest integer $w$ such that $P$ can be decomposed into $w$ chains of $P$.
Let us derive Proposition~\ref{prop:Heath-width} of Heath and Pemmaraju~\cite{HP-97} from such a chain partition.

\begin{proposition}\label{prop:Heath-width}
For every poset $P$, if $\width(P)\leq w$ then $\qn(P)\leq w^2$.
\end{proposition}
\begin{proof}
Let $P$ be a poset of width $w$ and $C_1,\ldots,C_w$ be a chain partition of $P$.
Let $L$ be any linear extension of $P$ and $a <_L b <_L c <_L d$ with $a\prec d$ and $b\prec c$.
Note that we must have either $a \parallel b$ or $c \parallel d$.
If follows that if $a \in C_i$, $b \in C_j$, $c \in C_k$, and $d \in C_\ell$, then $(i,\ell) \neq (j,k)$.
As there are only $w^2$ ordered pairs $(x,y)$ with $x,y \in [w]$, we can conclude that every nesting set of covers has cardinality at most $w^2$.
\end{proof}

Note that in the above proof $L$ is \textit{any} linear extension and that without choosing the linear extension $L$ carefully, upper bound $w^2$ is best-possible. Namely, if $P=\{a_1, \ldots, a_k,b_1, \ldots, b_k\}$ with comparabilities $a_i<b_j$ for all $1\leq i,j\leq k$, then $P$ has width $k$ and the linear extension $a_1< \ldots<a_k<b_k< \ldots< b_1$ creates a rainbow of size $k^2$.

We continue by showing that every poset of width $2$ has queue-number at most $2$, that is, we prove Theorem~\ref{thm:general-width-2}.

\begin{proof}[Theorem~\ref{thm:general-width-2}]
 Let $P$ be a poset of width $2$ and minimum element $0$ and $C_1,C_2$ be a chain partition of $P$. Note that the assumption of the minimum causes no loss of generality, since a $0$ can be added without increasing the width nor decreasing the queue-number. Any linear extension $L$ of $P$ partitions the ground set $X$ naturally into inclusion-maximal sets of elements, called \emph{blocks}, from the same chain in $\{C_1,C_2\}$ that appear consecutively along $L$, see Figure~\ref{fig:lazy-linear-extension}.
 We denote the blocks by $B_1,\ldots,B_k$ according to their appearance along $L$.
 We say that $L$ is \emph{lazy} if for each $i = 2,\ldots,k$, each element $x \in B_i$ has a relation to some element $y \in B_{i-1}$.
 A linear extension $L$ can be obtained by picking any minimal element $m\in P$, put it into $L$, and recurse on $P\setminus\{m\}$. Lazy linear extensions (with respect to $C_1,C_2$) can be constructed by the same process where additionally the next element is chosen from the same chain as the element before, if possible. 
 Note that the existence of a $0$ is needed in order to ensure the property of laziness with respect to $B_2$.
 
 \begin{figure}
  \centering
  \includegraphics{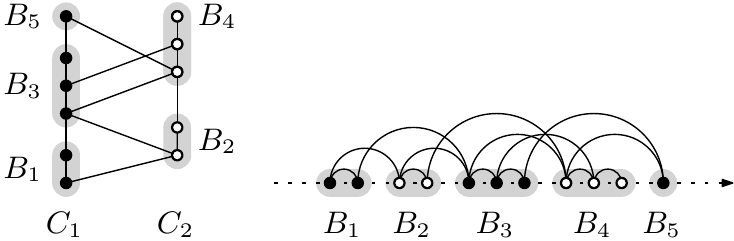}
  \caption{A poset of width~$2$ with a $0$ and a chain partition $C_1,C_2$ and the blocks $B_1,\ldots,B_5$ induced by a lazy linear extension with respect to $C_1,C_2$.}
  \label{fig:lazy-linear-extension}
 \end{figure}

 Now we shall prove that in a lazy linear extension no three covers are pairwise nesting.
 So assume that $a\prec b$ is any cover and that $a \in B_i$ and $b \in B_j$.
 As $L$ is lazy, $b$ is comparable to some element in $B_{j-1}$ (if $j \geq 2$) and all elements in $B_1,\ldots,B_{j-2}$ (if $j \geq 3$).
 With $a\prec b$ being a cover, it follows from $L$ being lazy that $i \in \{j-2,j-1,j\}$.
 If $i=j$, then no cover is nested under $a\prec b$.
 If $i=j-1$, then no cover $c\prec d$ is nested above $a\prec b$: either $c\in B_i$ and $d\in B_j$ and hence $c\prec d$ is not a cover, or both endpoints would be inside the same chain, i.e., $c,d$ are the last and first element of $B_{j-2}$ and $B_{j}$ or $B_{i}$ and $B_{i+2}$, respectively. This implies $c<_L a<_L d<_L b$ or $a<_L c<_L b<_L r$, respectively, and $c\prec d$ cannot nest above $a\prec b$.
 If $i=j-2$, then no cover is nested above $a\prec b$.
 Thus, either no cover is nested below $a\prec b$, or no cover is nested above $a\prec b$, or both.
 In particular, there is no three nesting covers and $\qn(P) \leq 2$.
\end{proof}

\begin{corollary}\label{cor:general-width}
 Every poset of width $w$ has queue-number at most $w^2-2\lfloor w/2 \rfloor$.
\end{corollary}

\begin{proof}
We take any chain partition of size $w$ and pair up chains to obtain a set $S$ of $\lfloor w/2 \rfloor$ disjoint pairs.
Each pair from $S$ induces a poset of width at most $2$, which by Theorem~\ref{thm:general-width-2} admits a linear order with at most two nesting covers. Let $L$ be a linear extension of $P$ respecting all these partial linear extensions. 

Now, following the proof of Proposition~\ref{prop:Heath-width} any cover can be labeled by a pair $(i,j)$ corresponding to the chains containing its endpoint. 
Thus, in a set of nesting covers any pair appears at most once, but for each $i,j$ such that $(i,j)\in S$ only two of the four possible pairs can appear simultaneously in a nesting.
This yields the upper bound.
\end{proof}

For an integer $k \geq 2$ we define a \emph{subdivided $k$-crown} as the poset $P_k$ as follows.
The elements of $P_k$ are $\{a_1,\ldots,a_k,b_1,\ldots,b_k,c_1,\ldots,c_k\}$ and the cover relations are given by $a_i \prec b_i$ and $b_i \prec c_i$ for $i = 2,\ldots,k$, $a_i \prec c_{i-1}$ for $i = 1,\ldots,k-1$, and $a_1 \prec c_k$; see the left of Figure~\ref{fig:P_k-example}.
We refer to the covers of the form $a_i \prec c_j$ as the \emph{diagonal covers} and we say that a poset $P$ has an \emph{embedded $P_k$} if $P$ contains $3k$ elements that induce a copy of $P_k$ in $P$ with all diagonal covers of that copy being covers of $P$.

 \begin{figure}
  \centering
  \includegraphics{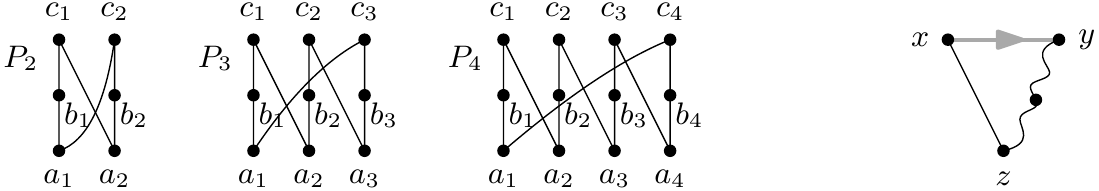}
  \caption{Left: The posets $P_2$, $P_3$, and $P_4$. Right: The existence of an element $z$ with cover relation $z \prec x$ and non-cover relation $z < y$ gives rise to a gray edge from $x$ to $y$.}
  \label{fig:P_k-example}
 \end{figure}

\begin{theorem}\label{thm:subdivided-crowns}
 If $P$ is a poset that for no $k \geq 2$ has an embedded $P_k$, then the queue-number of $P$ is at most the width of $P$.
\end{theorem}

\begin{proof}
 Let $P$ be any poset.
 For this proof we consider the cover graph $G_P$ of $P$ as a directed graph with each edge $xy$ directed from $x$ to $y$ if $x  \prec y$ in $P$.
 We call these edges the \emph{cover edges}.
 Now we augment $G_P$ to a directed graph $G$ by introducing for some incomparable pairs $x \parallel y$ a directed edge.
 Specifically, we add a directed edge from $x$ to $y$ if there exists a $z$ with $z < x,y$ in $P$ where $z \prec x$ is a cover relation and $z < y$ is not a cover relation; see the right of Figure~\ref{fig:P_k-example}.
 We call these edges the \emph{gray edges} of $G$.
 
 Now we claim that if $G$ has a directed cycle, then $P$ has an embedded subdivided crown.
 Clearly, every directed cycle in $G$ has at least one gray edge.
 We consider the directed cycles with the fewest gray edges and among those let $C = [c_1,\ldots,c_{\ell}]$ be one with the fewest cover edges.s
 First assume that $C$ has a cover edge (hence ${\ell} \geq 3$), say $c_1c_2$ is a gray edge followed by a cover edge $c_2c_3$.
 Consider the element $z$ with cover relation $z  \prec c_1$ and non-cover relation $z < c_2$ in $P$.
 By $z < c_2  \prec c_3$ we have a non-cover relation $z < c_3$ in $P$.
 Now if $c_1 \parallel c_3$ in $P$, then $G$ contains the gray edge $c_1c_3$ (see Figure~\ref{fig:P_k-proof-cases}(a)) and $[c_1,c_3,\ldots,c_{\ell}]$ is a directed cycle with the same number of gray edges as $C$ but fewer cover edges, a contradiction.
 On the other hand, if $c_1 < c_3$ in $P$ (note that $c_3 < c_1$ is impossible as $z  \prec c_1$ is a cover), then there is a directed path $Q$ of cover edges from $c_1$ to $c_3$ (see Figure~\ref{fig:P_k-proof-cases}(b)) and $C + Q - \{c_1c_2, c_2c_3\}$ contains a directed cycle with fewer gray edges than $C$, again a contradiction.
 
 \begin{figure}
  \centering
  \includegraphics{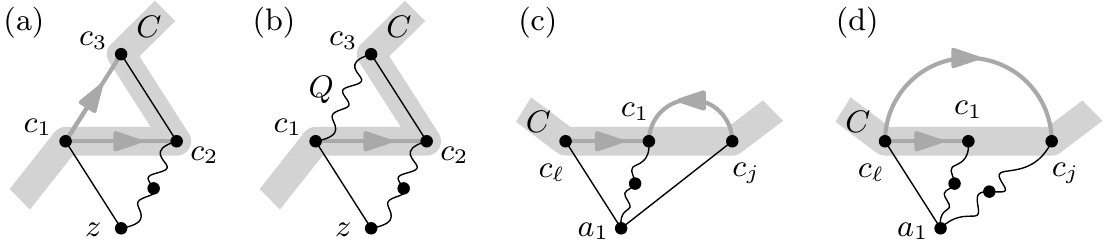}
  \caption{Illustrations for the proof of Theorem~\ref{thm:subdivided-crowns}.}
  \label{fig:P_k-proof-cases}
 \end{figure}
 
 Hence $C = [c_1,\ldots,c_{\ell}]$ is a directed cycle consisting solely of gray edges.
 Note that by the first paragraph $\{c_1,\ldots,c_{\ell}\}$ is an antichain in $P$.
 For $i=2,\ldots,{\ell}$ let $a_i$ be the element of $P$ with cover relation $a_i \prec c_{i-1}$ and non-cover relation $a_i < c_i$, as well as $a_1$ with cover relation $a_1 \prec c_{\ell}$ and non-cover relation $a_1 < c_1$.
 As $\{c_1,\ldots,c_{\ell}\}$ is an antichain and $a_i < c_i$ holds for $i=1, \ldots, {\ell}$, we have $\{c_1,\ldots,c_{\ell}\} \cap \{a_1,\ldots,a_{\ell}\} = \emptyset$.
 Let us assume that $a_1 < c_j$ in $P$ for some $j \neq 1,{\ell}$.
 If $a_1 \prec c_j$ is a cover relation, then there is a gray edge $c_jc_1$ in $G$ (see Figure~\ref{fig:P_k-proof-cases}(c)) and the cycle $[c_1,\ldots,c_j]$ is shorter than $C$, a contradiction.
 If $a_1 < c_j$ is a non-cover relation, then there is a gray edge $c_{\ell}c_j$ in $G$ (see Figure~\ref{fig:P_k-proof-cases}(d)) and the cycle $[c_j,\ldots,c_{\ell}]$ is shorter than $C$, again a contradiction.
 
 Hence, the only relations between $a_1,\ldots,a_{\ell}$ and $c_1,\ldots,c_{\ell}$ are cover relations $a_1 \prec c_{\ell}$ and $a_i \prec c_{i-1}$ for $i=2,\ldots,{\ell}$ and the non-cover relations $a_i < c_i$ for $i = 1,\ldots,\ell$.
 Hence $a_1,\ldots,a_{\ell}$ are pairwise distinct. Moreover, $\{a_1,\ldots,a_{\ell}\}$ is an antichain in $P$ since the only possible relations among these elements are of the form $a_1 < a_{\ell}$ or $a_i < a_{i-1}$, which would contradict that $a_1 \prec c_{\ell}$ and $a_i \prec c_{i-1}$ are cover relations.
 Finally, we pick for every $i=1,\ldots,{\ell}$ an element $b_i$ with $a_i < b_i < c_i$, which exists as $a_i < c_i$ is a non-cover relation.
 Together with the above relations between $a_1,\ldots,a_{\ell}$ and $c_1,\ldots,c_{\ell}$ we conclude that $b_1,\ldots,b_{\ell}$ are pairwise distinct and these $3{\ell}$ elements induce a copy of $P_{\ell}$ in $P$ with all diagonal covers in that copy being covers of $P$.
 
 Thus, if $P$ has no embedded $P_k$, then the graph $G$ we constructed has no directed cycles, and we can pick $L$ to be any topological ordering of $G$.
 As $G_P \subseteq G$, $L$ is a linear extension of $P$.
 For any two nesting covers $x_2 <_L x_1  <_L y_1 <_L y_2$ we have $x_1 \parallel x_2$ or $y_1 \parallel y_2$ or both, since $x_2 \prec y_2$ is a cover.
 However, if $x_2 < x_1$ in $P$, then there would be a gray edge from $y_2$ to $y_1$ in $G$, contradicting $y_1 <_L y_2$ and $L$ being a topological ordering of $G$.
 We conclude that $x_1 \parallel x_2$ and the left endpoints of any rainbow form an antichain, proving $\qn(P)\leq \width(P)$.
\end{proof}

Let us remark that several classes of posets have no embedded subdivided crowns, e.g., graded posets, interval orders (since these are 2+2-free, see~\cite{F-70}), or (quasi-)series-parallel orders (since these are N-free, see~\cite{HJ-85}). Here, 2+2 and N are the four-element posets defined by $a<b, c<d$ and $a<b, c<d, c<b$, respectively.
Also note that while subdivided crowns are planar posets, no planar poset with 0 and 1 has an embedded $k$-crown. Indeed, already looking at the subposet induced by the $k$-crown and the 0 and the 1, it is easy to see that there must be a crossing in any diagram.
Thus, we obtain:

\begin{corollary}\label{cor:interval}
 For any interval order, series-parallel order, and planar poset with 0 and 1, $P$ we have $\qn(P) \leq \width(P)$.
\end{corollary}



\section{Planar Posets of Bounded Width}
\label{sec:planar-posets-width}

Heath and Pemmaraju~\cite{HP-97} show that the largest queue-number among planar posets of width $w$ lies between $\lceil \sqrt{w} \rceil$ and $4w-1$.
Here we improve the lower bound to $w$ and the upper bound to $3w-2$.

\begin{proposition}\label{prop:planar-width-LB}
 For each $w$ there exists a planar poset $Q_w$ with 0 and 1 of width $w$ and queue-number $w$.
\end{proposition}

\begin{proof}
 We shall define $Q_w$ recursively, starting with $Q_1$ being any chain.
 For $w \geq 2$, $Q_w$ consists of a \emph{lower copy $P$} and a disjoint \emph{upper copy $P'$} of $Q_{w-1}$, three additional elements $a,b,c$, and the following cover relations in between:
 \begin{itemize}
  \item $a \prec x$, where $x$ is the 0 of $P$
  \item $y \prec x'$, where $y$ is the 1 of $P$ and $x'$ is the 0 of $P'$
  \item $y' \prec c$, where $y'$ is the 1 of $P'$
  \item $a \prec b \prec c$
 \end{itemize}
 It is easily seen that all cover relations of $P$ and $P'$ remain cover relations in $Q_w$, and that $Q_w$ is planar, has width $w$, $a$ is the 0 of $Q_w$, and $c$ is the 1 of $Q_w$.
 See Figure~\ref{fig:planar-width-LB} for an illustration.
 
 \begin{figure}
  \centering
  \includegraphics{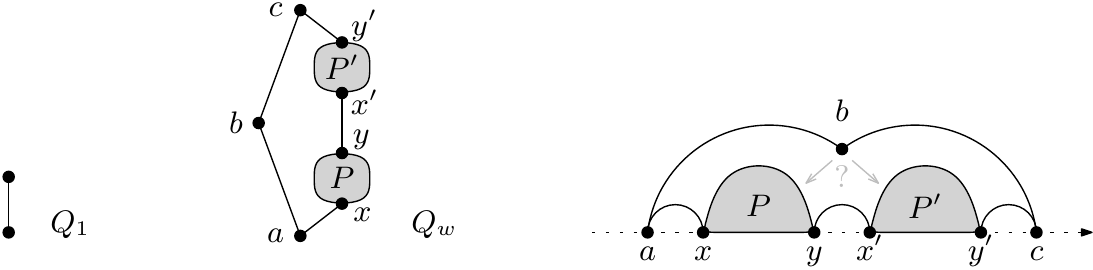}
  \caption{Recursively constructing planar posets $Q_w$ of width $w$ and queue-number $w$.
  Left: $Q_1$ is a two-element chain.
  Middle: $Q_w$ is defined from two copies $P,P'$ of $Q_{w-1}$.
  Right: The general situation for a linear extension of $Q_w$.}
  \label{fig:planar-width-LB}
 \end{figure}
 
 To prove that $\qn(Q_w) = w$ we argue by induction on $w$, with the case $w=1$ being immediate.
 Let $L$ be any linear extension of $Q_w$.
 Then $a$ is the first element in $L$ and $c$ is the last.
 Since $y \prec x'$, all elements in $P$ come before all elements of $P'$.
 Now if in $L$ the element $b$ comes after all elements of $P$, then $P$ is nested under cover $a\prec b$, and if $b$ comes before all elements of $P'$, then $P'$ is nested under cover $b\prec c$.
 We obtain $w$ nesting covers by induction on $P$ in the former case, and by induction on $P'$ in the latter case. This concludes the proof.
\end{proof}


%

Next we prove Theorem~\ref{thm:planar-width}, namely that the maximum queue-number of planar posets of width $w$ lies between $w$ and $3w-2$.

\begin{proof}[Theorem~\ref{thm:planar-width}]
 By Proposition~\ref{prop:planar-width-LB} some planar posets of width $w$ have queue-number $w$.
 So it remains to consider an arbitrary planar poset $P$ of width $w$ and show that $P$ has queue-number at most $3w-2$.
 To this end, we shall add some relations to $P$, obtaining another planar poset $Q$ of width $w$ that has a $0$ and $1$, with the property that $\qn(P) \leq \qn(Q) + 2w-2$.
 Note that this will conclude the proof, as by Corollary~\ref{cor:interval} we have $\qn(Q) \leq w$.
 
 Given a planar poset $P$ of width $w$, there are at most $w$ minima and at most $w$ maxima.
 Hence there are at most $2w-2$ extrema that are not on the outer face.
 For each such extremum $x$ --say $x$ is a minimum-- consider the unique face $f$ with an obtuse angle at $x$.
 We introduce a new relation $y < x$, where $y$ is a smallest element at face $f$, see Figure~\ref{fig:planar-width-UB}.
 Note that this way we introduce at most $2w-2$ new relations, and that these can be drawn y-monotone and crossing-free by carefully choosing the other element in each new relation. Furthermore, every inner face has a unique source and unique sink.
 
 \begin{figure}
  \centering
  \includegraphics{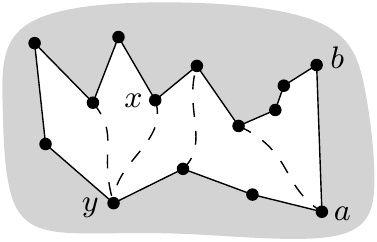}
  \caption{Inserting new relations (dashed) into a face of a plane diagram. Note that relation $a < b$ is a cover relation in $P$ but not in $Q$.}
  \label{fig:planar-width-UB}
 \end{figure}
 
 Now consider a cover relation $a \prec_P b$ that is not a cover relation in the new poset $Q$.
 For the corresponding edge $e$ from $a$ to $b$ in $Q$ there is one face $f$ with unique source $a$ and unique sink $b$.
 Now either way the other edge in $f$ incident to $a$ or to $b$ must be one of the $2w-2$ newly inserted edges, see again Figure~\ref{fig:planar-width-UB}.
 This way we assign $a \prec b$ to one of $2w-2$ queues, one for each newly inserted edge.
 Every such queue contains either at most one edge or two incident edges, i.e., a nesting is impossible, no matter what linear ordering is chosen later.
 
 We create at most $2w-2$ queues to deal with the cover relations of $P$ that are not cover relations of $Q$ and spend another $w$ queues for $Q$ dealing with the remaining cover relations of $P$.
 Thus, $\qn(P) \leq \qn(Q) + 2w - 2 \leq 3w-2$.
\end{proof}

\section{Planar Posets of Bounded Height}
\label{sec:planar-posets-height}

Recall Conjecture~\ref{conj:planar-height}, which states that every planar poset of height $h$ has queue-number at most $h$.
In the following, we give a counterexample to this conjecture:

\begin{proof}[Theorem~\ref{thm:planar-height-2}]
 Consider the graph $G$ that is constructed as follows:
 Start with $K_{2,10}$ with bipartition classes $\{a_1,a_2\}$ and $\{b_1,\ldots,b_{10}\}$.
 For every $i=1,\ldots,9$ add four new vertices $c_{i,1},\ldots,c_{i,4}$, each connected to $b_i$ and $b_{i+1}$.
 The resulting graph $G$ has $46$ vertices, is planar and bipartite with bipartition classes $X = \{b_1,\ldots,b_{10}\}$ and $Y = \{a_1,a_2\} \cup \{c_{i,j} \mid 1\leq i \leq 9, 1\leq j \leq 4\}$. See Figure~\ref{fig:planar-height-2-LB}.
 
 \begin{figure}
  \centering
  \includegraphics{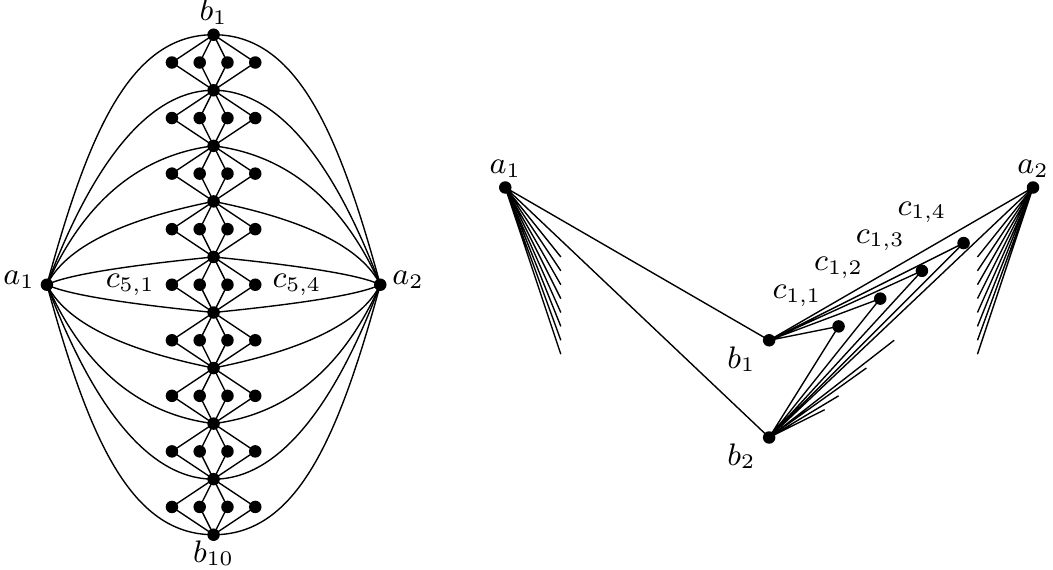}
  \caption{A planar poset $P$ of height $2$ and queue-number at least $4$. Left: The cover graph $G_P$ of $P$. Right: A part of a planar diagram of $P$.}
  \label{fig:planar-height-2-LB}
 \end{figure}
 
 Let $P$ be the poset arising from $G$ by introducing the relation $x < y$ for every edge $xy$ in $G$ with $x \in X$ and $y \in Y$.
 Clearly, $P$ has height $2$ and hence the cover relations of $P$ are exactly the edges of $G$.
 Moreover, by a result of Moore~\cite{Moo-75} (see also~\cite{BLR-90}) $P$ is planar because $G$ is planar, also see the right of Figure~\ref{fig:planar-height-2-LB}.
 
 We shall argue that $\qn(P) \geq 4$.
 To this end, let $L$ be any linear extension of $P$.
 Without loss of generality we have $a_1 <_L a_2$.
 Note that since in $P$ one bipartition class of $G$ is entirely below the other, any $4$-cycle in $G$ gives a $2$-rainbow.
 Let $b_{i_1},b_{i_2}$ be the first two elements of $X$ in $L$, $b_{j_1},b_{j_2}$ be the last two such elements.
 As $|X| = 10$ there exists $1\leq i \leq 9$ such that $\{i,i+1\} \cap \{i_1,i_2,j_1,j_2\} = \emptyset$, i.e., we have $b_{i_1},b_{i_2} <_L b_i,b_{i+1} <_L b_{j_1},b_{j_2} <_L a_1 <_L a_2$, where we use that $a_1$ and $a_2$ are above all elements of $X$ in $P$.
 
 Now consider the elements $C = \{c_{i,1},\ldots,c_{i,4}\}$ that are above $b_i$ and $b_{i+1}$ in $P$.
 As $|C| \geq 4$, there are two elements $c_1,c_2$ of $C$ that are both below $a_1,a_2$ in $L$, or both between $a_1$ and $a_2$ in $L$, or both above $a_1,a_2$ in $L$.
 Consider the $2$-rainbow $R$ in the $4$-cycle $[c_1,b_i,c_2,b_{i+1}]$.
 In the first case $R$ is nested below the $4$-cycle $[a_1,b_{i_1},a_2,b_{i_2}]$, in the second case the cover $b_{j_1}\prec a_1$ is nested below $R$ and $R$ is nested below the cover $b_{i_1}\prec a_2$, and in the third case $4$-cycle $[a_1,b_{j_1},a_2,b_{j_2}]$ is nested below $R$.
 As each case results in a $4$-rainbow, we have $\qn(P) \geq 4$.
\end{proof}

Even though Conjecture~\ref{conj:planar-height} has to be refuted in its strongest meaning, it might hold that planar posets of height $h$ have queue-number $O(h)$, or at least bounded by some function $f(h)$ in terms of $h$, or at least that planar posets of height $2$ have bounded queue-number.
As it turns out, all these statements are equivalent, and in turn equivalent to Conjecture~\ref{conj:queue-planar-graphs}.

\begin{theorem}\label{thm:equivalent-conjectures}
 The following statements are equivalent:
 \begin{enumerate}[label = (\roman*)]
  \item Planar graphs have queue-number $O(1)$ (Conjecture~\ref{conj:queue-planar-graphs}).\label{item:planar-graphs}
  \item Planar posets of height $h$ have queue-number $O(h)$.\label{item:height-linear}
  \item Planar posets of height $h$ have queue-number at most $f(h)$ for a function $f$.\label{item:height-any}
  \item Planar posets of height $2$ have queue-number $O(1)$.\label{item:height-2}
  \item Planar bipartite graphs have queue-number $O(1)$.\label{item:planar-bipartite}
 \end{enumerate}
\end{theorem}
\begin{proof}
 \begin{description}
  \item[\ref{item:planar-graphs}$\Rightarrow$\ref{item:height-linear}]
   Pemmaraju proves in his thesis~\cite{Pem-92} (see also~\cite{DPW-04}) that if $G$ is a graph, $\pi$ is a vertex ordering of $G$ with no $(k+1)$-rainbow, $V_1,\ldots,V_m$ are color classes of any proper $m$-coloring of $G$, and $\pi'$ is the vertex ordering with $V_1 <_{\pi'} \cdots <_{\pi'} V_m$, where within each $V_i$ the ordering of $\pi$ is inherited, then $\pi'$ has no $(2(m-1)k+1)$-rainbow.
   So if $P$ is any poset of height $h$, its cover graph $G_P$ has $\qn(G_P) \leq c$ by~\ref{item:planar-graphs} for some global constant $c >0$.
   Splitting $P$ into $h$ antichains $A_1,\ldots,A_h$ by iteratively removing all minimal elements induces a proper $h$-coloring of $G_P$ with color classes $A_1,\ldots,A_h$.
   As every vertex ordering $\pi'$ of $G$ with $A_1 <_{\pi'} \cdots <_{\pi'} A_h$ is a linear extension of $P$, it follows by Pemmaraju's result that $\qn(P) \leq 2(h-1)\qn(G_P) \leq 2ch$, i.e., $\qn(P) \in O(h)$.

  \item[\ref{item:height-linear}$\Rightarrow$\ref{item:height-any}$\Rightarrow$\ref{item:height-2}]
   These implications are immediate.
  
  \item[\ref{item:height-2}$\Rightarrow$\ref{item:planar-bipartite}]
   Moore proves in his thesis~\cite{Moo-75} (see also~\cite{BLR-90}) that if $G$ is a planar and bipartite graph with bipartition classes $A$ and $B$, and $P_G$ is the poset on element set $A \cup B = V(G)$ where $x<y$ if and only if $x \in A, y \in B, xy \in E(G)$, then $P_G$ is a planar poset of height $2$.
   As $G$ is the cover graph of $P_G$, we have $\qn(G) \leq \qn(P_G) \leq c$ for some constant $c >0$ by \ref{item:height-2}, i.e., $\qn(G) \in O(1)$.

  \item[\ref{item:planar-bipartite}$\Rightarrow$\ref{item:planar-graphs}] This is a result of Dujmovi\'c and Wood~\cite{DW-05}.
 \end{description}
 \vspace{-\baselineskip}
\end{proof}

Finally, we show that Conjecture~\ref{conj:planar-height} holds for planar posets with $0$ and $1$.

\begin{proof}[Theorem~\ref{thm:planar-height-0-and-1}]
 Let $P$ be a planar poset with $0$ and $1$.
 Then $P$ has dimension at most two~\cite{BFR-72}, i.e., it can be written as the intersection of two linear extensions of $P$.
 A particular consequence of this is, that there is a well-defined dual poset $P^\star$ in which two distinct elements $x,y$ are comparable in $P$ if and only if they are incomparable in $P^\star$.
 Poset $P^\star$ reflects a ``left of''-relation for each incomparable pair $x \parallel y$ in $P$ in the following sense:
 Any maximal chain $C$ in $P$ corresponds to a $0$-$1$-path $Q$ in $G_P$, which splits the elements of $P\setminus C$ into those left of $Q$ and those right of $Q$.
 Now $x <_{P^\star} y$ if and only if $x$ is left of the path for every maximal chain containing $y$ (equivalently $y$ is right of the path for every maximal chain containing $x$).
 Due to planarity, if $a\prec b$ is a cover in $P$ and $C$ is a maximal chain containing neither $a$ nor $b$, then $a$ and $b$ are on the same side of the path $Q$ corresponding to $C$. In particular, if for $x,y\in C$ we have $a <_{P^\star} x$ and $b \parallel y$, then $b$ and $y$ are comparable in $P^\star$, but if $y <_{P^\star} b$ we would get a crossing of $C$ and $a\prec b$. Also see the left of Figure~\ref{fig:planar-height-UB}. We summarize:

 \begin{enumerate}[label = ($\star$)]
  \item If $a\prec b$, $a <_{P^\star} x$ for some $x \in C$ and $b \parallel y$ for some $y \in C$, then $b <_{P^\star} y$. \label{enum:same-side}
 \end{enumerate}   
 
 \begin{figure}
  \centering
  \includegraphics{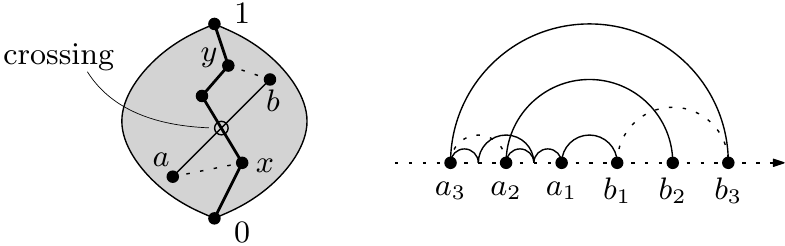}
  \caption{Left: Illustration of~\ref{enum:same-side}: If $a <_{P^\star} x$, $b \parallel y$, $x < y$, and $a\prec b$ is a cover, then $b <_{P^\star} y$ due to planarity. Right: If $a_3 <_L a_2 <_L a_1 <_L b_1 <_L b_2 <_L b_3$ is a $3$-rainbow with $a_2,a_3 < a_1$, then $a_3 < a_2$.}
  \label{fig:planar-height-UB}
 \end{figure}
 
 Now let $L$ be the \emph{leftmost} linear extension of $P$, i.e., the unique linear extension $L$ with the property that for any $x \parallel y$ in $P$ we have $x <_L y$ if and only if $x < y$ in $P^\star$. 
 Assume that $a_2 <_L a_1 <_L b_1 <_L b_2$ is a pair of nesting covers $a_1\prec b_1$ below $a_2\prec b_2$.
 Then $a_1 \parallel a_2$ (hence $a_2 <_{P^\star} a_1$) or $b_1 \parallel b_2$ (hence $b_1 <_{P^\star} b_2$) or both.
 Observe that the latter case is impossible, as for any maximal chain $C$ containing $a_1\prec b_1$ we would have $a_2 <_{P^\star} a_1$ with $a_1 \in C$ and $b_1 <_{P^\star} b_2$ with $b_1 \in C$, contradicting \ref{enum:same-side}.
 So the nesting of $a_1\prec b_1$ below $a_2\prec b_2$ is either of type~A with $a_2 < a_1$, or of type~B with $b_1 < b_2$.
 See Figure~\ref{fig:leftmost-LE}.
 
 \begin{figure}
  \centering
  \includegraphics{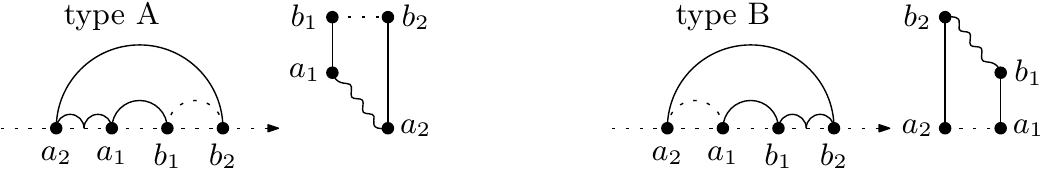}
  \caption{A nesting of $a_1\prec b_1$ below $a_2\prec b_2$ of type~A (left) and type~B (right).}
  \label{fig:leftmost-LE}
 \end{figure}

 Now consider the case that cover $a_2\prec b_2$ is nested below another cover $a_3\prec b_3$, see the right of Figure~\ref{fig:planar-height-UB}.
 Then also $a_1\prec b_1$ is nested below $a_3\prec b_3$ and we claim that if both, the nesting of $a_1\prec b_1$ below $a_2\prec b_2$ as well as the nesting of $a_1\prec b_1$ below $a_3\prec b_3$, are of type~A (respectively type~B), then also the nesting of $a_2\prec b_2$ below $a_3\prec b_3$ is of type~A (respectively type~B).
 Indeed, assuming type~B, we would get $a_3 <_{P^\star} a_2$ and $b_1 <_{P^\star} b_3$, which together with any maximal chain $C$ containing $a_2 < a_1 <b_1$ contradicts~\ref{enum:same-side}.

 Finally, let $a_k <_L \cdots <_L a_1 <_L b_1 <_L \cdots <_L b_k$ be any $k$-rainbow and let $I = \{ i \in [k] \mid a_i < a_1 \}$, i.e., for each $i \in I$ the nesting of $a_1\prec b_1$ below $a_i\prec b_i$ is of type~A.
 Then we have just shown that the nesting of $a_j\prec b_j$ below $a_i\prec b_i$ is of type~A whenever $i,j \in I$ and of type~B whenever $i,j \notin I$.
 Hence, the set $\{a_i \mid i \in I\} \cup \{a_1,b_1\} \cup \{b_i \mid i \notin I\}$ is a chain in $P$ of size $k+1$, and thus $k \leq h-1$.
 It follows that $P$ has queue-number at most $h-1$, as desired. 
\end{proof}

The proof of the following can be found in the appendix.

\begin{proposition}\label{prop:planar-height-LB}
 For each $h$ there exists a planar poset $Q_h$ of height $h$ and queue-number $h-1$.
\end{proposition}
\section{Conclusions}
We studied the queue-number of (planar) posets of bounded height and width. Two main problems remain open: bounding the queue-number by the width and bounding it by a function of the height in the planar case, where the latter is equivalent to the central conjecture in the area of queue-numbers of graphs. For the first problem the biggest class known to satisfy it are posets without the embedded the subdivided $k$-crowns for $k\geq 2$ as defined in Section~\ref{sec:general-posets}. Note, that proving it for $k\geq 3$ would imply that Conjecture~\ref{conj:general-width} holds for all $2$-dimensional posets, which seems to be a natural next step. 

Let us close the paper by recalling another interesting conjecture from~\cite{HP-97}, which we would like to see progress in:
\begin{conjecture}[Heath and Pemmaraju~\cite{HP-97}]\label{conj:planar-vertices}{\ \\}
 Every planar poset on $n$ elements has queue-number at most $\lceil \sqrt{n} \rceil$.
\end{conjecture}

\bibliographystyle{abbrv}
\bibliography{lit}

\section{Appendix}

\begin{proof}[Proposition~\ref{prop:planar-height-LB}]
  We shall recursively define a planar poset $Q_h$ of height $h$ and queue-number $h-1$, together with a certain set of marked subposets in $Q_h$.
 Each marked subposet consists of three elements $x,y,z$ forming a \emph{V-subposet} in $Q_h$, i.e., $y < x,z$ but $x \parallel z$, with both relations $y < x$ and $y < z$ being cover relations of $Q_h$, and $y$ being a minimal element of $Q_h$.
 We call such a marked subposet in $Q_h$ a V-poset.
 Finally, we ensure that the V-posets are pairwise incomparable, namely that any two elements in distinct V-posets are incomparable in $Q_h$.
 
 For $h=2$ let $Q_2$ be the three-element poset as shown in left of Figure~\ref{fig:planar-height-LB}, which also forms the only V-poset of $Q_2$.
 Clearly $Q_2$ has height~$2$ and queue-number~$1$.
 For $h \geq 3$ assume that we already constructed $Q_{h-1}$ with a number of V-posets in it.
 Then $Q_h$ is obtained from $Q_{h-1}$ by replacing each V-poset by the eight-element poset shown in the right of Figure~\ref{fig:planar-height-LB}, which introduces (for each V-poset) five new elements.
 Moreover, two new V-posets are identified in $Q_h$ as illustrated in Figure~\ref{fig:planar-height-LB}.
 
 \begin{figure}
  \centering
  \includegraphics{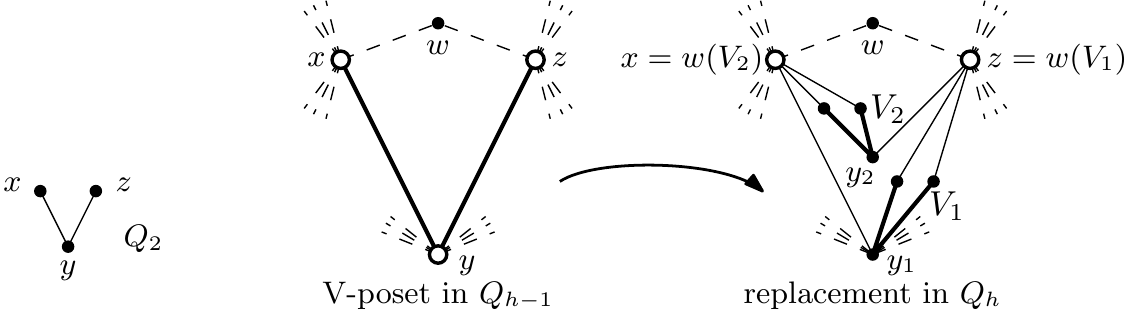}
  \caption{Constructing planar posets of height $h$ and queue-number $h-1$.
   Left: $Q_2$ is a three-element poset and its only V-poset.
   Right: $Q_h$ is recursively defined from $Q_{h-1}$ by replacing each V-poset by an eight-element poset and identifying two new V-posets.}
  \label{fig:planar-height-LB}
 \end{figure}

 It is easy to check that $Q_h$ is planar and has height $h$, since $Q_{h-1}$ has height $h-1$ and the V-posets in $Q_{h-1}$ are pairwise incomparable.
 Moreover, every V-poset in $Q_h$ contains a minimal element of $Q_h$ and all V-posets in $Q_h$ are pairwise incomparable.
 Finally, observe that, as long as $h \geq 3$, for every V-poset $V$ in $Q_h$ there is a unique smallest element $w = w(V)$ that is larger than all elements in $V$, see the right of Figure~\ref{fig:planar-height-LB}.

 In order to show that $\qn(Q_h) \geq h-1$, we shall show by induction on $h$ that for every linear extension $L$ of $Q_h$ there exists a $(h-1)$-rainbow in $Q_h$ with respect to $L$ whose innermost cover is contained in a V-poset $V$ of $Q_h$, and, if $h \geq 3$, whose second innermost cover has the element $w(V)$ as its upper end.
 This clearly holds for $h=2$.
 For $h \geq 3$, consider any linear extension $L$ of $Q_h$.
 This induces a linear extension $L'$ of $Q_{h-1}$ as follows:
 The set $X$ of elements in $Q_h$ not contained in any V-poset is also a subset of the elements in $Q_{h-1}$.
 The remaining elements of $Q_{h-1}$ are the minimal elements of the V-posets in $Q_{h-1}$.
 For each minimal element $y$ of $Q_{h-1}$ consider the two corresponding V-posets in $Q_h$ with its two corresponding minimal elements $y_1,y_2$.
 Let $\hat{y} \in \{y_1,y_2\}$ be the element that comes first in $L$, i.e., $\hat{y} = y_1$ if and only if $y_1 <_L y_2$.
 Then we define $L'$ to be the ordering of $Q_{h-1}$ induced by the ordering of $X \cup \{\hat{y} \mid y \in Q_{h-1} - X\}$ in $L$.
 Note that $L'$ is a linear extension of $Q_{h-1}$, even though $X \cup \{\hat{y} \mid y \in Q_{h-1} - X\}$ does not necessarily induce a copy of $Q_{h-1}$ in $Q_h$.

 By induction on $Q_{h-1}$ there exists a $(h-2)$-rainbow $R$ with respect to $L'$ whose innermost cover is contained in a V-poset $V$ and, provided that $h-1 \geq 3$, its second innermost cover has $w = w(V)$ as its upper end.
 Consider the elements $x,y,z$ of $V$ with $y$ being the minimal element, and the two corresponding V-posets $V_1,V_2$ with minimal elements $y_1,y_2$ of $Q_h$, where $y_1x$ and $y_2z$ are covers; see Figure~\ref{fig:planar-height-LB}.
 By definition of $\hat{y}$ and $L'$, all elements of $\{x,y\} \cup V_1 \cup V_2$ lie between $\hat{y}$ (included) and $w$ (excluded, if $h-1 \geq 3$) with respect to $L$.
 
 Assume without loss of generality that $x <_L z$.
 If $y_2 <_L y_1$ ($\hat{y} = y_2$), then the V-poset with $y_1$ is nested completely under the cover $y_2z$ and replacing in $R$ the innermost cover by the cover $y_2z$ and any cover with $y_1$ gives a $(h-1)$-rainbow with the desired properties.
 If $y_1 <_L y_2$ ($\hat{y} = y_1$), then the V-poset with $y_2$ is nested completely under the cover $y_1x$ and replacing in $R$ the innermost cover by the cover $y_1x$ and any cover with $y_w$ gives a $(h-1)$-rainbow with the desired properties, which concludes the proof.
\end{proof}

\end{document}